\documentclass[11pt]{article}

\usepackage[T1]{fontenc}
\usepackage{amsmath}
\usepackage{amsfonts}
\usepackage{array}
\usepackage{enumerate}
\usepackage{graphicx}
\usepackage{amsthm}
\usepackage{xcolor}
\usepackage[utf8x]{inputenc}
\usepackage{graphicx}
\usepackage{caption}
\usepackage{subcaption}
\usepackage{amsmath}
\usepackage{amssymb}
\usepackage{changebar}
\usepackage{color}
\usepackage{latexsym,amsfonts,amscd,amsthm}
\usepackage{fancyhdr}
\usepackage{chemarr}
\usepackage{tikz}
\usepackage{cite}
\usepackage{changes}

\numberwithin{equation}{section}

\newtheorem{theorem}{Theorem}[section]
\newtheorem{lemma}[theorem]{Lemma}
\newtheorem{proposition}[theorem]{Proposition}
\newtheorem{corollary}[theorem]{Corollary}

\theoremstyle{definition}

\newtheorem{remark}{Remark}

\DeclareMathOperator{\rank}{rank}
\DeclareMathOperator{\im}{im}

\newcommand{\abs}[1]{\ensuremath{\left\vert#1\right\vert}}
\newcommand{\norm}[2][\relax]{\ifx#1\relax \ensuremath{\left\Vert#2\right\Vert} \else \ensuremath{\left\Vert#2\right\Vert_{#1}}\fi}

\begin{document}

\title{A coordinate-independent version of Hoppensteadt's convergence theorem}

\author{ Christian Lax,  Katrin Seliger, Sebastian Walcher \\ Lehrstuhl A f\"ur Mathematik\\ RWTH Aachen\\ 52056 Aachen, Germany}

\maketitle 

\begin{abstract} The classical theorems about singular perturbation reduction (due to Tikhonov and Fenichel) are concerned with convergence on a compact time interval (in slow time) as a small parameter approaches zero. For unbounded time intervals Hoppensteadt gave a convergence theorem, but his criteria are generally not easy to apply to concrete given systems. We state and prove a convergence result for autonomous systems on unbounded time intervals which relies on criteria that are relatively easy to verify, in particular for the case of a one-dimensional slow manifold. As for applications, we discuss several reaction equations from biochemistry.
 
\end{abstract}

\begin{align*}
 &\text{\textbf{MSC2010:} 34E15, 92C45, 34C45}\\
&\text{\textbf{Keywords:} singular perturbations, reduction, reaction system, Lyapunov}
\end{align*}

\section{Introduction}

Singular perturbation phenomena occur frequently in the modelling and analysis of chemical or biological systems, in particular for reaction equations, and are highly relevant for reducing the dimension of a problem. For reaction equations (involving a small parameter $\varepsilon$), such phenomena may often be interpreted in the context of quasi-steady state (QSS) or partial equilibrium approximations (PEA). In many instances, the classical work of Tikhonov \cite{tikh} and Fenichel \cite{fenichel} provides a method to obtain a reduced equation. \\
The theorems of Tikhonov and Fenichel guarantee convergence on some fixed compact time interval as $\varepsilon\to 0$.  But beyond this result, in many applications one expects convergence for all positive times after a short initial phase, i.e. with slow time ranging in $[\tau_0,\infty)$ for some $\tau_0>0$. (In general convergence does not hold on an unbounded interval; see Fenichel \cite{fenichel}, p.~68 for a well-known example involving the van der Pol equation).  \\ 
Hoppensteadt \cite{Hoppensteadt} stated and proved a convergence theorem for singularly perturbed systems which guarantee convergence on unbounded intervals, essentially resolving the matter up to coordinate transformations. However, in many potential applications these transformations (which effect a separation of variables into ``slow'' and ``fast'') cannot be determined explicitly, and the hypotheses of the theorem are difficult to verify. In fact, even if a system is given in slow-fast coordinates, Hoppensteadt's crucial conditions may not be readily verifiable. In the literature one finds some applications of this theorem where an explicit coordinate transformation is determined and the validity of Hoppensteadt's conditions is verified directly. Thus, Cavallo and Natale \cite{cavnat}, Teel et al. \cite{teeletal}, and Back and Shim \cite{bashi} discuss applications to control theory. For the classical Michaelis-Menten enzyme model (with low enzyme concentration), no doubt was ever expressed about the validity of the reduction for all positive times, and the direct estimates given in Segel and Slemrod \cite{SSl} do imply convergence for the case of irreversible product formation. For the reversible case, it seems that a convergence proof was given only relatively recently, in \cite{nw11}. This proof uses Hoppensteadt's criteria, and the crucial part invokes explicit knowledge of a first integral for the fast system. The argument cannot be extended to familiar variants of Michaelis-Menten, e.g. those including inhibition or cooperativity. One purpose of the present work is to provide more easily applicable criteria for reaction systems.\\

The paper is organized as follows. We start the main section (Section 2) with an auxiliary result on Lyapunov functions and asymptotic stability. Then we proceed to a version of Hoppensteadt's theorem for autonomous systems that are written in slow-fast coordinates with special properties. Following this, we do not only specialize (and thus simplify) Hoppensteadt's conditions for autonomous systems but we also replace one of the conditions with another that is somewhat stronger, but readily verifiable.  We next recall Tikhonov-Fenichel reduction for singularly perturbed systems with no a priori separation of slow and fast variables \cite{gw2}. Finally we give additional conditions which guarantee convergence to solutions of the reduced system on unbounded intervals, with Theorem \ref{hshsatz} the main result. Some of the conditions we impose (e.g. eigenvalue conditions) are relatively easy to verify in applications, but for others verification may still be problematic. In particular this concerns the existence of a global parameterization of the asymptotic slow manifold, and the existence of a Lyapunov function for the reduced system on this manifold. But at least for the case of a one-dimensional slow manifold, which is highly relevant for QSS in biochemistry, these problems can be resolved completely, and one obtains readily applicable criteria.  In Section 3 we discuss a number of examples. \\
Some of the results presented are based on work in the theses \cite{selma} and \cite{laxdiss}.
The Appendix (Section 5) contains a list of Hoppensteadt's conditions and the corresponding theorem for easy reference.

\section{Hoppensteadt's theorem for autonomous systems}
The goal of this section is to state and prove a version of Hoppensteadt's convergence theorem \cite{Hoppensteadt} for autonomous systems which is readily applicable to the investigation of a reasonably large class of differential equations. In the main result we will not require any a priori separation of fast and slow variables, and we will focus on conditions that are amenable to explicit verification. \\
We first prove an auxiliary result on Lyapunov functions, and then an autonomous version of the convergence theorem for special slow-fast coordinates, before turning to a general coordinate-free version. The result is, in particular, easily applicable to systems with a one-dimensional slow manifold.\\
Let $U\subset \mathbb R^m$ be open, and $p\in C^1(U;\mathbb R^m)$. We consider the differential equation
\begin{equation}\label{odep}
x^\prime=p(x)
\end{equation}
on $ U$, with the prime (here and in the following) denoting differentiation with respect to the independent variable $\tau$.
\subsection{An auxiliary result}

\begin{lemma}\label{lyapunov}
Let $Y\subset \mathbb R^m$ be a submanifold and $\widehat K\subset U$ compact such that $Y\cap\widehat  K$ is positively invariant with respect to \eqref{odep}. Assume there exists a neighborhood $S\subset U$ of $Y\cap \widehat K$ and a function $\varphi\in C^1(S)$ that satisfies the following conditions:
  \begin{enumerate}[(i)]
   \item The inequality $\varphi(x)\geq0$ holds for all $x\in Y\cap \widehat K$, and there exists exactly one $z\in Y\cap\widehat  K$ such that $\varphi(z)=0$.
   \item Given a norm $\norm{\cdot }$ on $\mathbb R^m$ there exist $c_1,c_2>0$, a positive integer $a$, and $\rho>0$ such that for all $x\in Y\cap B_{\rho}(z)$ the inequalities
    \[
     c_1\norm{x-z}^a\leq \varphi(x)\leq c_2 \norm{x-z}^a
   \]
 are satisfied. (Here $z$ is from (i), and $B_{\rho}(z)$ denotes the open ball with center $z$ and radius $\rho$.)
   \item There exist  $\nu>0$ and $k\geq 1$ such that the Lie derivative of $\varphi$ with respect to $p$ satisfies
    \[
     L_{p}(\varphi)(x)\leq -\nu \varphi(x)^k
    \]
   for all $x\in Y\cap \widehat K$. (Recall  $L_p(\varphi)(x)=D\varphi(x)\cdot p(x)$ for all $x$.)
  \end{enumerate}
 Then there exists $c>0$ such that for all $x_0\in Y\cap \widehat K$, $x_0\not=z$ the solution $\Phi(\tau,x_0)$ of the initial value problem $x^\prime=\frac{dx}{d\tau}=p(x)$, $x(0)=x_0$ satisfies the inequality
  \[
   \norm{\Phi(\tau,x_0)-z}\leq c\norm{x_0-z}\gamma(\tau),
  \]
with 
\[
\gamma(\tau)=\left\{ \begin{array}{lcl} e^{-\nu \tau/a} &\text{ for }& k=1,\\
                                                       ((k-1)\nu \tau \varphi(x_0)^{k-1}+1)^{1/[a(1-k)]}&\text{ for }& k>1
\end{array}\right. 
\]
strictly decreasing to $0$ as $\tau\to\infty$. \\

\end{lemma}

\begin{proof}
 The function
  \[
   \frac{\varphi(x)}{\norm{x-z}^a}
  \]
 is continuous on $\widehat K\cap \{x\in\mathbb R^m;\, \norm{x-z}\geq \rho\}$  and therefore bounded below and above by positive constants. Hence there exist $0<c_1^*<c_2^*$ such that for all  $x\in Y\cap \widehat K$ the inequalities
  \begin{equation}\label{cineq}
   c_1^*\norm{x-z}^a\leq \varphi(x)\leq c_2^*\norm{x-z}^a
   \end{equation}
 hold. \\
Now use condition (iii) and recall a result on differential inequalities (e.g. Amann \cite{Amann}, Lemma 16.4): Since the initial value problem
\[
w^\prime = -\nu w^k,\quad w(0)= \varphi(x_0)>0
\]
in $\mathbb R$ is solved by
\[
\widetilde \gamma(\tau)=\left\{ \begin{array}{lcl} \varphi(x_0)e^{-\nu \tau} &\text{ for }& k=1,\\
                                                       ((k-1)\nu \tau +\varphi(x_0)^{1-k})^{1/(1-k)}&\text{ for }& k>1,
\end{array}\right.
\]
the solution $\Phi(\tau,x_0)$ satisfies
  \[
   \varphi(\Phi(\tau,x_0))\leq \widetilde \gamma(\tau)\leq c_2^*\norm{x_0-z} ^a  \gamma(\tau)
  \]
due to (ii). By virtue of \eqref{cineq} we obtain
  \[
   \norm{\Phi(\tau,x_0)-z}\leq \left(\tfrac1{c_1^*}\varphi(\Phi(\tau,x_0))\right)^{1/a}\leq \left(\tfrac{c_2^*}{c_1^*}\right)^{1/a}\norm{x_0-z}\gamma(\tau).
  \]
 The assertion follows.
\end{proof}

\subsection{Systems in Tikhonov standard form}\label{subs22}
Here we will prove an intermediate result for systems written in special coordinates. With the exception of (ASII) below, the conditions are patterned after Hoppensteadt \cite{Hoppensteadt}, conditions (I) through (VII).\\
In the following denote by $\abs{\cdot}_1$ the 1-norm,  let $s$ and $r$ be positive integers and $m=s+r$. For $R>0$ we define
  \begin{equation}\label{tubeone}
\begin{array}{rcl}
   S_R&:=&\{y=(y_1,y_2)\in\mathbb R^{s+r},\ \abs{y}_1=\abs{y_1}_1+\abs{y_2}_1\leq R\},\\
S_{1,R}&:=&\{y_1\in\mathbb R^{s},\, \abs{y_1}_1\leq R\},\\
S_{2,R}&:=&\{y_2\in\mathbb R^{r},\,\abs{y_2}_1\leq R\}.\\
\end{array}
  \end{equation}
We consider a singularly perturbed autonomous system in Tikhonov standard form
  \begin{align}
   &y_1'= f(y_1,y_2,\varepsilon)\label{tnfhs1}\\
   &y_2'=\varepsilon^{-1}g(y_1,y_2,\varepsilon)\label{tnfhs2}
  \end{align} 
as well as its counterpart in fast time $t=\tau/\varepsilon$, viz.
 \begin{align}
   &\dot y_1= \varepsilon f(y_1,y_2,\varepsilon)\\
   &\dot y_2=g(y_1,y_2,\varepsilon)
  \end{align} 
 subject to the following conditions:
  \begin{enumerate}
\item[(AS0)] Both $f$ and $g$ are $C^2$ functions in an open subset $\widetilde U$ of $\mathbb R^{s}\times \mathbb R^r\times \mathbb R$, and $\widetilde U$ contains $S_R\times [0,\varepsilon_0)$ for some $R>0$ and $\varepsilon_0>0$.
   \item[(ASI)] The system 
    \begin{align}
     &y_1'=f(y_1,y_2,0) \label{redSystem1}\\
     &0=g(y_1,y_2,0)\label{redSystem2}
    \end{align}
 admits the stationary point $0\in\mathbb R^s\times \mathbb R^r$.
\item[(ASII)] For all $y_1\in S_{1,R}$ one has $g(y_1,\,0,\,0)=0$, and there is a constant $\nu>0$ such that all eigenvalues of $D_2g(y_1,0,0)$ have real part $\leq -\nu$.\\ (Here and in the following $D_i$ denotes the partial derivative with respect to $y_i$.)

   \item[(ASIII)] The hypotheses of Lemma \ref{lyapunov} hold with $p(y)=f(y_1,0,0)$, $Y=\{(y_1,0)\in S_R\}$ and $z=0$.
  \end{enumerate}
Invoking a compactness argument, it would suffice in (ASII) to require that the eigenvalues of $D_2g(y_1,0,0)$ have real part $<0$ for all $y_1\in S_{1,R}$.

\begin{proposition}\label{hopstandard} Whenever assumptions {\em (AS0)} through {\em (ASIII)}  hold, there exists a compact neighborhood $K\times [0,\varepsilon_0^*]\subseteq S_R\times [0,\varepsilon_0)$ of $0\in\mathbb R^{s+r+1}$ with the following properties: Given  $y_0:=(y_{1,0},\,y_{2,0})\in K$ and $\varepsilon\in  (0,\varepsilon_0^*]$, the solution $\Phi(\tau,y_0)$ of \eqref{tnfhs1}--\eqref{tnfhs2}
exists for $0\leq \tau<\infty$. As $\varepsilon\to 0$, this solution converges uniformly on all closed subsets of $(0,\infty)$ to the solution 
of \eqref{redSystem1}--\eqref{redSystem2} with respect to the initial value $y_1(0)=y_{1,0}$.
\end{proposition}

\begin{proof} One has to verify  conditions (I) through (VII) in Hoppensteadt \cite{Hoppensteadt}; for the reader's convenience these are recalled in Section \ref{appsec} below. Clearly (AS0) and (ASI), together with the fact that the system is autonomous, ensure that conditions (I) through (V) hold. (In the autonomous case the uniformity requirements follow readily by continuity and compactness.) Condition (VI) is a consequence of (ASIII) and Lemma \ref{lyapunov}. \\
There remains to show that (ASII) implies the validity of (VII). In principle one could refer to Fenichel \cite{fenichel}, but we give a proof with some details here.
\begin{enumerate}[(i)]
\item  From $g(y_1,0,0)=0$ for all $y_1\in S_{1,R}$ and by Hadamard's Lemma (see e.g. Nestruev \cite{nestruev2003smooth}, Lemma 2.8) there exists a $C^1$ function $\widehat R$ such that
\[
g(y_1,y_2,0)=\widehat R(y_1,y_2)y_2
\]
for all $(y_1,y_2)\in S_R$. Furthermore
\[
\widehat R(y_1,y_2)y_2=\left(A(y_1)+R(y_1,y_2)\right)y_2
\]
where $y_1\mapsto A(y_1)=\widehat R(y_1,0) \in \mathbb R^{r\times r}$ is $C^1$ and 
$\|R(y_1,y_2)\|\to 0$ as $\|y_2\|\to 0$, uniformly in $y_1$. According to (ASII),  for $y_1\in S_{1,R}$ all eigenvalues of $A(y_1)$ have real part $\leq -\nu<0$.
\item We denote by $\phi(\cdot,\cdot)$ the standard Euclidean scalar product on $\mathbb R^r$, and thus have 
\[
\phi(y_2,y_2)=\|y_2\|_2^2.
\]
Denote by $C$ the unit sphere in $\mathbb R^r$ with respect to $\|\cdot\|_2$. For every $y_1$ there exists a $\theta(y_1)>0$ such that
\[
\begin{array}{rccl}
2\phi(y_2,A(y_1)y_2)&\leq & -2\theta(y_1)\|y_2\|_2^2& \text{  for all  }y_2\in\mathbb R^r,\\
2\phi(y_2,A(y_1)y_2)&\leq & -2\theta(y_1)& \text{  for all  }y_2\in C.
\end{array}
\]
The proof of the first inequality follows by the arguments in Walter \cite{Walter} (\S 30, IV(d) and proof of \S 29, VIII). These imply that there exists some positive definite symmetric bilinear form $\psi$ such that 
\[
2\psi(y_2,A(y_1)y_2)\leq -\nu/2\cdot\psi(y_2,y_2) \text{  for all  }y_2,
\]
and the assertion follows by the equivalence of all norms on $\mathbb R^r$. The second inequality is a simple consequence but it shows that one can choose $-2\theta(y_1)$ as the maximum of the left hand side function on $C$.
\item Given $y_1^*\in S_{1,R}$ there exists a neighborhood $U(y_1^*)$ such that 
\[
\phi(y_2,A(y_1)y_2)\leq -\theta(y_1^*)\|y_2\|_2^2 \text{  for all  }y_2\in\mathbb R^r, \,y_1\in U(y_1^*).
\]
This follows by the estimates in (ii), the continuity of the map 
\[
S_{1,R}\times C\to \mathbb R,\quad (y_1,y_2)\mapsto \phi(y_2, A(y_1)y_2),
\]
and the homogeneity of $\phi$. In conjunction with the compactness of $S_{1,R}$ this estimate implies the existence of some $\beta>0$ such that
\[
\phi(y_2,A(y_1)y_2)\leq -2\beta \|y_2\|_2^2 \text{  for all  } (y_1,y_2)\in S_R.
\]
\item Moreover there exists $\rho>0$ such that 
\[
|2\phi(y_2,R(y_1,y_2)y_2)|\leq \beta\cdot\|y_2\|_2^2 \text{  for all  } y_1\in S_{1,R},\,\,y_2\in S_{2,\rho}
\]
due to uniform  convergence with respect to $y_1$ (again, see Walter \cite{Walter}, loc.cit.). Altogether one obtains
\[
L_g(\phi)(y_2)= 2\phi(y_2, A(y_1)y_2)+2\phi(y_2,R(y_1,y_2)y_2)\leq -\beta\|y_2\|_2^2
\]
for all $(y_1,y_2)\in S_{1,R}\times S_{2,\rho}$.
\item Denote the solution of $ x_2^\prime=g(y_1,x_2,0)$ with initial value $y_2$ by $\Gamma(y_1,y_2,\tau)$. Then by Amann \cite{Amann}, Lemma 16.4 one finds 
\[
 \phi(\Gamma(y_1,y_2,\tau))\leq \|y_2\|_2^2\cdot\exp(-\beta \tau),
\] 
hence
\[
\|\Gamma(y_1,y_2,\tau)\|\leq \|y_2\|\cdot\exp(-\beta \tau/2).
\]
 Hoppensteadt's condition (VII) follows via the equivalence of all norms on $\mathbb R^r$.
 \end{enumerate}
\end{proof}

\begin{remark} Hoppensteadt \cite{Hoppensteadt} develops his conditions (for non-autonomous systems) from a more general setting in a step-by-step manner, with certain normalizations (that can not necessarily be carried out explicitly) being invoked at various stages; see Section \ref{appsec}.   Eventually his crucial conditions require the special setting which we consider above (and furthermore restrict to autonomous systems). 
\end{remark}

\subsection{General systems}
The goal of this subsection is to extend Proposition \ref{hopstandard} to settings where no a priori separation of ``slow'' and ``fast'' variables is given. Thus we consider a system 
 \begin{equation}\label{evolution}
   \dot x=h(x,\varepsilon)=h^{(0)}(x)+\varepsilon h^{(1)}(x)+\varepsilon^2 h^*(x,\varepsilon)
  \end{equation}
with right-hand side $C^2$ in $(x,\varepsilon)$, and $(x,\varepsilon)$ is in some open subset of $\mathbb R^m\times\mathbb R$ that contains $(x_0,0)$ for some $x_0$. We will also work with the time-scaled version ($\tau=\varepsilon t$ as in subsection \ref{subs22}), thus
\begin{equation}\label{evolutionscaled}
 x^\prime=  \frac{d x}{d\tau}=\varepsilon^{-1}h(x,\varepsilon)=\varepsilon^{-1}h^{(0)}(x)+ h^{(1)}(x)+\ldots
  \end{equation}
of this equation. We first recall a coordinate-free version of standard (Tikhonov-Fenichel) singular perturbation reduction from \cite{gw2}, Theorem 1. (The theorem was stated for systems with rational right-hand side, but as noted in \cite{gw2}, Remark 2, sufficient differentiability already guarantees existence.) The following conditions are relevant.
\begin{itemize}
\item[(TF0)]  There exists a point $x_0$ in the zero set $\mathcal V(h^{(0)})$ such that $\rank Dh^{(0)}(x)=r<m$ for all $x\in \mathbb R^m$ in some neighborhood of $x_0$.
\end{itemize}
\begin{remark}
By the implicit function theorem, (TF0) implies the existence of a neighborhood $U$ of $x_0$ such that $V:=U\cap \mathcal V(h^{(0)})$ is a $(m-r)$-dimensional submanifold. 
\end{remark}
\begin{itemize}
\item[(TFI)] There is  a direct sum decomposition
  \[
   \mathbb R^m=\ker Dh^{(0)}(x) \oplus \im Dh^{(0)}(x)
  \]
for all $x\in V$.
(In other words, one requires that algebraic and geometric multiplicity of the eigenvalue zero of $Dh^{(0)}(x)$ are equal.)
\end{itemize}

For details and proofs concerning the next two results we refer to \cite{gw2}.
\begin{proposition}\label{decompred} Let {\em (TF0)} and {\em (TFI)} be given. Then the following hold.
\begin{enumerate}[(a)]
\item (Product decomposition) On some neighborhood $\widetilde U\subseteq U$ of $x_0$ there exist $C^1$ maps  \[P\colon \widetilde U\to \mathbb R^{m\times r}\quad \text{and} \quad \mu\colon\widetilde U\to \mathbb R^r\]  with $\rank P(x_0)=\rank D\mu(x_0)=r$, such that
    \[
     h^{(0)}(x)=P(x)\mu(x),\quad x\in \widetilde U.
    \]
 Moreover, the zero set $Y$ of $\mu$ satisfies $Y=V\cap \widetilde U= \mathcal V(h^{(0)})\cap\widetilde U$. The entries of $\mu$ may be taken as any $r$ entries of $h^{(0)}$ that are functionally independent at $x_0$.
\item The system
    \begin{equation}\label{Grenzsystem}
     x^\prime=q(x):=Q(x)\cdot h^{(1)}(x)
    \end{equation}
  with \[Q(x):=Id-P(x)(D\mu(x)P(x))^{-1}D\mu(x),\] 
is defined in $\widetilde U$, and the manifold $Y$ is an invariant set of \eqref{Grenzsystem}. Moreover, every entry of $\mu$ is a first integral of \eqref{Grenzsystem}.
\end{enumerate}
\end{proposition}
We will call \eqref{Grenzsystem} the Tikhonov-Fenichel reduction of \eqref{evolutionscaled}.  The result holds for every connected component of $Y$, hence we may and will assume that $Y$ is connected.

\begin{remark} \label{remarkred}
\begin{enumerate}[(a)]
\item Conditions (TF0) and (TFI) ensure the existence of a coordinate transformation that puts \eqref{evolutionscaled} into {Tikhonov} standard form, and the reduced system \eqref{Grenzsystem} corresponds to the familiar reduction with slow and fast variables; see  \cite{nw11}.
\item As was shown in \cite{gw2}, for rational $h^{(0)}$ one may choose $P$ and $\mu$ rational, and the decomposition can be obtained constructively by methods of algorithmic algebra.
\end{enumerate}
\end{remark}

The next condition guarantees local convergence of solutions to solutions of the reduced system.
\begin{itemize}
\item[(TFII)]  All nonzero eigenvalues of $Dh^{(0)}(x)$, $x\in Y$,  have negative real part.
\end{itemize}
With these assumptions one can state a coordinate-free local version of (Tikhonov's and) Fenichel's reduction theorem; see \cite{gw2}, Theorem 1.

\begin{proposition}\label{zitatgoeke} 
Assume that {\em (TF0)}, {\em (TFI)}  and {\em (TFII)} hold. Then there exists $T>0$ and a neighborhood $U^*\subset U$ of $Y$ such that solutions of \eqref{evolutionscaled} starting in $U^*$ converge uniformly on $[\tau_0,T]$ to solutions of the reduced system \eqref{Grenzsystem} on $Y$ as $\varepsilon\to 0$, for any $\tau_0$ with $0<\tau_0<T$.
\end{proposition}

\begin{remark}\label{zitatremark}
 \begin{enumerate}[(a)]
  \item The submanifold $V$ is called the {\em asymptotic slow manifold} (or {\em critical manifold}). 
  \item Concerning the question of finding the appropriate initial values on $Y$ (which was in principle also settled by Fenichel \cite{fenichel}, Theorem 9.1), we briefly summarize the discussion in \cite{gw2} Proposition 2: The system $\dot x=h^{(0)}(x)$ admits $m-r$ independent first integrals in a neighborhood of $x_0$, and the intersection of a common level set of the first integrals with $Y$ consists (locally) of a single point. To project an initial value of system \eqref{evolution} to an initial value of \eqref{Grenzsystem} on $Y$, choose the corresponding intersection point. Thus, a solution of \eqref{evolutionscaled} starting at $x_0\in U^*$ converges to the solution of \eqref{Grenzsystem} starting at the projected initial value.
  \item In the situation of Proposition \ref{zitatgoeke} we sometimes call \eqref{Grenzsystem} a \textit{convergent} Tikhonov-Fenichel reduction of \eqref{evolutionscaled}; in contrast to a {\em formal} reduction whenever only (TF0) and (TFI) hold.
  \item The proof of Proposition \ref{zitatgoeke} (see \cite{nw11} Proposition 2.3 and \cite{gw2} Theorem 1) shows that the coordinate transformation which puts \eqref{evolutionscaled} into {Tikhonov} standard form (see Remark \ref{remarkred} (a)) also maps solutions of the reduced system \eqref{Grenzsystem} to solutions of the corresponding reduced system \eqref{redSystem1}--\eqref{redSystem2} in {Tikhonov} standard form.
 \end{enumerate}
\end{remark}

Up to this point we focussed on conditions which ensure convergence of singular perturbation reduction on some compact subinterval of $(0, \infty)$. We now introduce additional conditions to guarantee validity of the reduction on unbounded intervals. The first of these conditions could be weakened, but it is convenient for applications and it is satisfied for many relevant systems, in particular reaction systems.

\begin{itemize}
\item [(CIS)] There exists a compact neighborhood $K\subseteq \widetilde U$ of $x_0$ which is positively invariant for all differential equations \eqref{evolution} with $0<\varepsilon<\varepsilon_0$.
\end{itemize}

By continuous dependence one obtains:

\begin{lemma} 
 Under the assumptions of Proposition \ref{zitatgoeke} and given condition (CIS), the set $K\cap Y$ is  positively invariant for the reduced system \eqref{Grenzsystem}.
\end{lemma}

Next come the crucial conditions.

\begin{itemize}
\item [(GP)] There exists a contractible open subset $W$ of $\mathbb R^{s}$, $s=m-r$, and a global injective $C^2$ immersion $\Lambda^*\colon W\to Y$.
\end{itemize}

\begin{remark}
\begin{enumerate}[(a)]
\item We introduce condition (GP) to make the reasoning more transparent, and in order to state the following Lemma in a more general context. But below we will introduce a further condition (LC) which actually implies (GP). Indeed, as an argument in the proof of Theorem \ref{hshsatz} will show, by global asymptotic stability there exists a flow on $W$ which contracts to a point (see also Corollary \ref{graphcor}).
\item We (may and) will assume that $S_{1,R}\subseteq W$ for some $R>0$.
\item In the statement of (GP) one may replace $Y$ by a relatively open neighborhood of $Y\cap K$ in $Y$.
\end{enumerate}
\end{remark}

\begin{lemma}
Under the assumptions of Proposition \ref{zitatgoeke} and given conditions (CIS) and (GP) there is a compact set $K^*\supseteq K\cap Y$ with nonempty interior, some $0<\rho\leq R$ and a $C^2$-diffeomorphism
\[
\Lambda\colon S_\rho\to K^*, \quad \text{with  } \Lambda|_{S_{1,\rho}\times \{0\}}= \Lambda^*|_{S_{1,\rho}\times \{0\}}.
\]
(Thus, there are open neighborhoods of $S_\rho$ resp. $K^*$ that are mapped to each other by $\Lambda$ and its inverse.)
\end{lemma}

\begin{proof}
The normal bundle $N$ of $Y$ is trivial, since $W$ is simply connected; see Hirsch \cite{hirsch}, Ch.~4, Corollary 2.5.  Now the assertion follows by injectivity of $\Lambda^*$ and Hirsch \cite{hirsch}, Ch.~4, Theorem 5.1.
\end{proof}

The final condition we require is as follows.

\begin{itemize}
\item [(LC)] The reduced system  \eqref{Grenzsystem} admits one and only one stationary point $z$ in $Y\cap K$, and the conditions in Lemma \ref{lyapunov} are satisfied. 
\end{itemize}

With these assumptions  the convergence statement of Proposition \ref{hopstandard} carries over.

\begin{theorem}\label{hshsatz}
For system \eqref{evolutionscaled} assume that {\em(TF0)}--{\em(TFII)} as well as {\em(CIS)}, {\em(GP)} and {\em(LC)} are satisfied.
 Then there exist a compact $\widetilde K\subseteq \Lambda(S_\rho)$ with nonempty interior,  $z\in \widetilde K$, and $\varepsilon_0^*>0$ with the following properties: Given  $y_0\in \widetilde K$ and $0<\varepsilon<\varepsilon_0^*$, the solution $\Phi(t,y_0)$ of \eqref{tnfhs1}--\eqref{tnfhs2}
exists for $0\leq t<\infty$. As $\varepsilon\to 0$, this solution converges uniformly on all closed subsets of $(0,\infty)$ to the solution 
of \eqref{Grenzsystem} with initial value according to Remark \ref{zitatremark} b).

\end{theorem}

\begin{proof}
Use $\Lambda\colon S_\rho\to K^*$ to define
\[
\widetilde h(x,\varepsilon):=D\Lambda(x)^{-1}h(\Lambda(x,\varepsilon)).
\]
By construction, the diffeomorphism $\Lambda$ sends solutions of $\dot x =\widetilde h(x,\varepsilon)$ to solutions of $\dot x=h(x,\varepsilon)$; and it is sufficient to verify conditions (AS0) through (ASIII) for the former system.\\
\begin{enumerate}[(i)]
\item For $\varepsilon=0$ one has the identity
\[
D\Lambda(x)\widetilde h^{(0)}(x)=h^{(0)}(\Lambda(x)),
\]
with $\widetilde h^{(0)}(x):=\widetilde h(x,0)$. Let $x=(x_1,x_2)$, with $x_1\in \mathbb R^s$ and $x_2\in\mathbb R^r$. Since $\Lambda$ extends $\Lambda^*$, we have 
\[
\widetilde h^{(0)}((x_1,0))=D\Lambda((x_1,0))^{-1}h^{(0)}(\Lambda((x_1,0)))=0, 
\]
and we obtain the conjugacy property
\[
D\widetilde h^{(0)}((x_1,0))=D\Lambda((x_1,0))^{-1}Dh^{(0)}(\Lambda((x_1,0)))D\Lambda((x_1,0))
\]
by differentiation, noting that the second term on the right hand side vanishes due to $h^{(0)}(\Lambda((x_1,0)))=0$.
\item Moreover  $\dot x=\widetilde h^{(0)}(x)$ may be assumed to be in the particular form
\[
\begin{array}{rcl}
\dot x_1&=&0\\
\dot x_2&=& g^{(0)}(x_1,x_2)
\end{array}
\]
with $ g^{(0)}(x_1,0)=0$. In other words, $\dot x =\widetilde h(x,\varepsilon)$ is in Tikhonov standard form and (AS0), (ASI) hold.\\
To verify this, note that (TFII) holds and use Fenichel \cite{fenichel}, Lemma 5.3.
(A different proof for analytic systems makes use of the fact that the differential equation  $\dot x =\widetilde h^{(0)}(x)$ admits $s$ independent first integrals in the neighborhood of any stationary point; see \cite{nw11}, Proposition 2.2.)
\item By conjugacy of Jacobians, (TFII) and compactness one sees that (ASII) is satisfied.
\item There remains to verify the existence of a Lyapunov function so that (ASIII) holds. The reduced system corresponding to $\widetilde h$ will be called $x^\prime =\widetilde q(x)$. It has the special form
\[
\begin{array}{rcl}
x_1^\prime&=& f^{(1)}(x_1,0)\\
x_2^\prime &=& 0
\end{array}
\] 
with the slow manifold being given by $x_2=0$. Due to Remark \ref{zitatremark} (d), the map $\Lambda$ sends solutions of $\dot x=\widetilde q(x)$ to solutions of $\dot x=q(x)$. Now the Lyapunov function $\varphi$ for $q$ satisfies $L_q(\varphi)\leq -\nu\cdot\varphi^k$, and with the well-known identity
\[
L_{\widetilde q}(\varphi\circ\Lambda)=L_q(\varphi)\circ\Lambda
\]
one obtains that $\widetilde\varphi:=\varphi\circ\Lambda$ is a Lyapunov function for $\widetilde q$, and that the inequality
\[
L_{\widetilde q}(\widetilde \varphi)\leq -\nu \widetilde \varphi^k
\]
holds.  Moreover, obviously $\widetilde\varphi\geq 0$ with $0$ the only zero. Finally, since $\Lambda^{-1}$ is a diffeomorphism and its derivative is bounded on the compact set $S_\rho$, the mean value estimate shows the existence of positive constants $k_1$ and $k_2$ such that
\[
   k_1\norm{x-z}\leq \norm{\Lambda^{-1}(x)-\Lambda^{-1}(z)}\leq k_2\norm{x-z}
\]
for all $x\in S_\rho$. This implies condition (ii) from Lemma \ref{lyapunov}.
\end{enumerate}

\end{proof}

\begin{corollary}\label{graphcor} If $Y$ is the graph of some smooth function  $\Gamma\colon W\to \mathbb R^r$, with $W\subseteq R^{m-r}$ contractible and open, and {\em(CIS)} and {\em (LC)} are satisfied in addition to {\em(TF0)}--{\em(TFII)}, then {\em (GP)} and thus the conclusion of Theorem \ref{hshsatz} hold.
\end{corollary}

\begin{remark}\label{lyappar} 
In the setting of Theorem \ref{hshsatz}, it suffices to require the existence of a Lyapunov function $\varphi$ for $q$ on $Y\cap K$ (rather than in some neighborhood of $Y$), since $\widetilde \varphi$ can be extended to $S_\rho$ by setting $\widehat \varphi(x_1,\,x_2):=\widetilde\varphi(x_1)$.
\end{remark}

\subsection{One-dimensional slow manifolds}
For a general system \eqref{evolutionscaled} the verification of condition (LC) on $Y$ a priori requires explicit knowledge of a Lyapunov function. However, for one-dimensional slow manifolds a simple condition will imply (LC). We first note a property of differential equations on real intervals which is essentially common knowledge (since differential equations in $\mathbb R$ are gradient systems); a proof is included for the reader's convenience.
\begin{lemma}\label{1dlyapunov}
 Let $U\subseteq \mathbb R$ be an open interval containing $0$, and $p\in C^1(U)$ with $p(0)=0$. Moreover let $K\subset U$ be compact with $0\in K$, and $0$ the only stationary point of $x'=p(x)$ in $K$. If $0$ is linearly asymptotically stable (i.e. $p'(0)<0$) then 
  \[
   \widetilde\varphi(x)=-\int_{0}^xp(y)\: dy
  \]
is a Lyapunov function of $x'=p(x)$ which satisfies the hypotheses of Lemma \ref{lyapunov} on $K$, with $a=2$ and $k=1$.
\end{lemma}

\begin{proof} By Hadamard's lemma
\[
p(x)= x\cdot \widehat p(x),\text{  with  }\widehat p \text {  continuous and   }\widehat p(0)=-\theta<0;
\]
moreover $\widehat p$ is negative throughout $K$.
This implies that $K$ is positively invariant and that $\widetilde \varphi$ is nonnegative, with $0$ its only zero. In some neighborhood $\widetilde U$ of $0$ one has the estimates
\[
-2\theta\cdot x\leq p(x)\leq -\theta/2\cdot x \text{  for }x>0,\quad -\theta/2\cdot x\leq p(x)\leq -2\theta\cdot x \text{  for  }x<0,
\]
which imply
\[
\theta/2\cdot x^2\leq \widetilde\varphi(x)\leq 2\theta\cdot x^2
\]
for all $x\in \widetilde U$. Therefore condition (ii) from Lemma \ref{lyapunov} holds on $K$, since $\widetilde\varphi$ and $L_p(\widetilde\varphi)$ are continuous and the complement of $\widetilde U$ in $K$ is compact. By construction 
$L_{p}(\widetilde\varphi)(x)=-p^2(x)$, and the above estimate shows
\[
L_{p}(\widetilde\varphi)(x)=-x^2 \cdot 4\theta^2
\]
in $\widetilde U$, whence conditions (ii) and (iii) in Lemma \ref{lyapunov} hold in $\widetilde U$ with $a=2$ and $k=1$, and (by compactness and continuity arguments) on all of $K$.

\end{proof}

Now we can state our result.

\begin{proposition}\label{1dprop} 
For system \eqref{evolutionscaled} assume that {\em(TF0)}--{\em(TFII)} and {\em(CIS)} are satisfied. Moreover assume that $Y\cap K$ is one-dimensional, connected, contains exactly one stationary point $z$, and that the linearization of the reduced equation $\dot x=q(x)$ at $z$ admits a negative eigenvalue. 
Then the conclusion of Theorem \ref{hshsatz} holds.
\end{proposition}

\begin{proof} We first note that the one-dimensional compact and connected manifold is homeomorphic to a compact interval or to a circle; see e.g. Milnor \cite{milnor}. But the latter is incompatible with the existence of a single stationary point that is asymptotically stable. The curve $Y$ admits a global parameterization by curve length, and therefore (GP) is satisfied. According to Proposition \ref{decompred}, there are $m-1$ functionally independent defining equations for $Y$ near any of its points, and these are first integrals for the reduced equation (see also Proposition \ref{zitatgoeke}). Therefore the linearization at $z$ admits the eigenvalue $0$ with geometric multiplicity $\geq m-1$ at $z$, and the eigenspace of the nonzero (negative) eigenvalue must be equal to the tangent space to $Y$ at $z$. Now Lemma \ref{1dlyapunov} and Theorem \ref{hshsatz}  apply.
\end{proof}

\section{Examples}

In this section we will discuss some reaction equations, with an emphasis on one-dimensional slow manifolds. Let $\mathbb R_+^n$ be the set of all vectors $x\in\mathbb R^n$ with nonnegative entries. Moreover, the concentration of a chemical species $Z$ will be denoted with a lowercase letter $z$.

\subsection{Michaelis-Menten reaction and variants}

 The well-known (reversible) Michaelis-Menten reaction is defined by the reaction scheme
   \[
    E+S \xrightleftharpoons[k_{-1}]{k_{1}} C \xrightleftharpoons[k_{-2}]{k_{2}} E+P;
  \]  
see Michaelis and Menten\cite{michaelismenten}, and also Briggs and Haldane \cite{briggshaldane}. The concentrations of each chemical species will be denoted by the corresponding lower-case letter. By mass-action kinetics, using the linear first integrals $e+c$ and $s+c+p$ from stoichiometry and assuming that initially no complex $C$ or product $P$ are present, one obtains the following two-dimensional problem:
  \begin{align*}
   &\dot s=-k_1e_0s+(k_1s+k_{-1})c\\
   &\dot c=k_1e_0s-(k_1s+k_{-1}+k_2)c+k_{-2}(e_0-c)(s_0-s-c).
  \end{align*}
\subsubsection{Small enzyme concentration}
 The standard approach takes the assumption of small initial enzyme concentration $e_0=\varepsilon e_0^*$. By the results in \cite{nw11} and \cite{gswz}, there exists a convergent Tikhonov-Fenichel reduction to
  \[
   s^\prime=-e_0^*\frac{(k_1k_2+k_{-1}k_{-2})s-k_{-1}k_{-2}s_0}{k_1s+k_{-1}+k_2+k_{-2}(s_0-s)},
  \]
 the slow manifold being $V=\{(s,0)\in {\mathbb R}^2_+\}$.  Using an explicit transformation to Tikhonov standard form and employing some straightforward but elaborate computations, convergence on unbounded time intervals was already proven in \cite{nw11}. (As mentioned before, Segel and Slemrod \cite{SSl} gave a proof by direct estimates for the irreversible case $k_{-2}=0$.) We use this example only to illustrate how Proposition \ref{1dprop} greatly simplifies convergence proofs.
Indeed, the validity of  (TF0)--(TFII) and (CIS) is easy to verify, and clearly, the reduced equation admits exactly one stationary point
  \[
   s^*=\frac{k_{-1}k_{-2}s_0}{k_1k_2+k_{-1}k_{-2}}
  \]
 which is linearly asymptotically stable. Thus, Proposition \ref{1dprop} shows convergence (with the irreversible case included for $k_{-2}=0$.)

\subsubsection{Slow product formation for the irreversible system}
 There are other choices for a ``small parameter''  that yield convergent Tikhonov-Fenichel reductions of the Michaelis-Menten reaction (see \cite{gwz} for an exhaustive discussion). All of these admit one-dimensional slow manifolds and it is easy to check the validity of the hypotheses of Proposition \ref{1dprop}. For instance, assuming slow, irreversible product formation (i.e., $k_2=\varepsilon k_2^*$ small and $k_{-2}=0$), one finds the reduction
\[
 s^\prime = \frac{-(k_1s+k_{-1})k_1k_2^*e_0s}{k_1k_{-1}e_0+(k_1s+k_{-1})^2},
\]
on $V=\{(s,c)\in{\mathbb R}^2_+,\ k_1e_0s=(k_1s+k_{-1})c\}$. Again, Proposition \ref{1dprop} is applicable to show convergence on $[\tau_0,\,\infty)$ for any $\tau_0>0$.  

\subsubsection{Competitive inhibition}\label{compinone}
As an extension of the Michaelis-Menten model we discuss an irreversible enzyme reaction with inhibition (see e.g. Keener and Sneyd \cite{KeenerSneyd}). The reaction scheme is given by
  \begin{align*}
   &E+S \xrightleftharpoons[k_{-1}]{k_{1}} C_1 \xrightarrow{k_{2}} E+P\\
   &E+I \xrightleftharpoons[k_{-3}]{k_{3}} C_2. 
  \end{align*}
We assume the usual initial values $c_1(0)=c_2(0)=0$, $s(0)=s_0>0$, $e(0)=e_0>0$ and $i(0)=i_0>0$. Using the linear first integrals 
  \begin{align*}
   &\psi_1(e,s,c_1,c_2,p,i)=e+c_1+c_2,\quad \psi_2(e,s,c_1,c_2,p,i)=s+c_1+p,\\ &\psi_3(e,s,c_1,c_2,p,i)=i+c_2
  \end{align*} 
from stoichiometry, one obtains a three-dimensional system. Again we assume that the initial enzyme concentration is low, thus $e_0=\varepsilon e_0^*$, and obtain the differential equation 
\[
  \begin{array}{cclcl}
   \dot s&=&k_{-1}c_1+k_1s(c_1+c_2)&-&\varepsilon e_0^* k_1s \\
  \dot c_1&=&-k_1s(c_1+c_2)-(k_{-1}+k_2)c_1 &+& \varepsilon e_0^* k_1s\\
  \dot c_2&=&-k_3(c_1+c_2)(i_0-c_2)-k_{-3}c_2&+&\varepsilon e_0^*k_3(i_0-c_2)
  \end{array}
\]
The reduction was computed in \cite{gswz}, subsection 3.2; in particular (TF0)-(TFII) (GP) and (CIS) hold for the slow manifold
$Y$ defined by $c_1=c_2=0$, and the reduced equation 
\[
s^\prime=-e_0^*\, \frac{k_1k_2k_{-3}s}{k_{-3}(k_1s+k_{-1}) +(k_{-1}+k_2)k_3i_0 +k_2k_{-3}}
\]
admits the only stationary point $0$, which is linearly asymptotically stable. By  Proposition \ref{1dprop} we obtain convergence on every interval $[\tau_0,\,\infty)$, with $\tau_0>0$. (The same holds true for reversible product formation.) \\ For this system the method employed in \cite{nw11} is not feasible, since an explicit transformation to Tikhonov standard form (in particular the requisite first integrals) seems to be unavailable.

\subsection{Maltose transport} 
In order to further illustrate the range of applicability of Proposition \ref{1dprop}, we discuss an example which is less straightforward from a computational perspective. Thus, we continue the discussion in \cite{gw2}, Section 4, of a reaction equation proposed by Stiefenhofer \cite{sti} for maltose transport. According to the model, in order to pass through the cell membrane, a maltose molecule $X$ first reacts with a binding protein $Z$ to a complex $Y_1$. The latter reacts with the membrane-bound receptor $R$, forming a complex $Y_2$, which subsequently degrades, releasing maltose into the cell. This last process is modelled by a reaction involving the maltose concentration $X_i$ in the interior of the cell. Moreover, Stiefenhofer assumes a direct reaction between the binding protein and the membrane receptors, modelled by a further reaction. Altogether, the transport mechanism is modelled
by the network
  \begin{align*}
   &Y_2  \xrightarrow{k_{1}} R+Z+X_i,\quad Z+X\xrightleftharpoons[k_{-2}]{k_{2}} Y_1\\
   &Y_1+R\xrightleftharpoons[k_{-3}]{k_{3}} Y_2,\quad Z+R\xrightleftharpoons[k_{-4}]{k_{4}} Y_3.
  \end{align*}
In order to reduce notational and computational complexity, we follow Stiefenhofer by setting all rate constants equal to to 1, except for $k_1=\varepsilon$. Moreover we define $v:=(x,z,r,\xi,y_1,y_2,y_3)$, $\bar v:=(\xi,y_1,y_2,y_3)$ and assume $y_1(0)=y_2(0)=y_3(0)=0$. Now we can use the stoichiometric first integrals 
  \begin{align*}
   &\psi_1(v)=z+y_1+y_2+y_3,\quad \psi_2(v)=r+y_2+y_3,\quad \psi_3(v)=x+\xi+y_1+y_2
  \end{align*}
to write the reaction rates in the form 
  \begin{align*}
   \hat E_1(v)&=-y_2=:E_1(\bar v),\\
  \hat E_2(v)&=y_1-zx\\ &=y_1-(z_0-(y_1+y_2+y_3))(x_0+\xi_0-(\xi+y_1+y_2))=:E_2(\bar v)\\
 \hat  E_3(v)&=y_2-y_1r\\ &=y_2-y_1(r_0-(y_2+y_3))=:E_3(\bar v)\\
  \hat E_4(v)&=y_3-zr\\ &=y_3-(z_0-(y_1+y_2+y_3))(r_0-(y_2+y_3))=:E_4(\bar v).
  \end{align*}
Thus the reaction is described by the following system
  \begin{align*}
   &\dot \xi=-\varepsilon E_1(\bar v)\\
   &\dot y_1=-E_2(\bar v)+E_3(\bar v)\\
   &\dot y_2=\varepsilon E_1(\bar v)-E_3(\bar v)\\
   &\dot y_3=-E_4(\bar  v).
  \end{align*}
 As proven in \cite{gw2}, there exists a (formal) reduction to
  \begin{align*}
   &\xi^\prime=y_2\\
   &y_1^\prime=\frac{y_2(y_1+y_2+y_3-z_0)}{n(\bar v)}\\
   &y_2^\prime=-y_2-\frac{y_2(\xi-\xi_0+2(y_1+y_2)+y_3-(x_0+z_0+1))}{n(\bar v)}\\
   &y_3^\prime=\frac{y_2((y_2+y_3)(y_1+y_2+y_3-r_0-z_0)+r_0(z_0-y_1))}{n(\bar v)}
  \end{align*}
 where
  \[
   n(\bar v)=\xi_0-\xi+(y_1+y_2+y_3-z_0)(y_2+y_3-r_0-1)-(y_1+y_2)+1+x_0.
  \]
 The system is given on $K:=L\cap Y$, where
  \[
   L:=\{\bar v\in{\mathbb R}^4_+,\ y_1+y_2+y_3\leq z_0,\ y_2+y_3\leq r_0,\ \xi+y_1+y_2\leq \xi_0+x_0\}
  \]
 is the chemically relevant region (determined by stoichiometry) and the curve
  \[
   Y:=\{\bar v\in{\mathbb R}^4_+, E_2(\bar v)=E_3(\bar v)=E_4(\bar v)=0\}
  \]
 is the slow manifold.\\
We first complete the discussion in  \cite{gw2} by showing that all nonzero eigenvalues of the Jacobian have negative real parts; in particular we have convergence of the reduction. To this end, note that with $\varepsilon=0$ the Jacobian of the reaction equation can be written as
\[
\begin{pmatrix} 0&0&0&0\\
                  *&-1-a-b-c& -a-b+1+d&-a+d\\
                *& c& -1-d&-d\\
                 *& -c&-b-c& -1-b-c\\
\end{pmatrix}
\]
with
\[
\begin{array}{rcl}
a&:=&x_0+\xi_0-(\xi+y_1+y_2)\\
b&:=&z_0-(y_1+y_2+y_3)\\
c&:=&r_0-(y_2+y_3)\\
d&:=& y_1
\end{array}
\]
all of which are nonnegative in view of the first integrals $\psi_1$, $\psi_2$ and $\psi_3$ and nonnegativity of concentrations. The characteristic polynomial 
\[
x^3+A_1x^2+A_2x+A_3
\]
of the lower right $3\times 3$ minor has all roots with negative real parts if (and only if) $A_1>0$, $H_2:=A_1A_2-A_3>0$ and $A_3>0$; see the Hurwitz-Routh criterion (Gantmacher \cite{gantmacher}, Ch.~V, \S6). A straightforward computation (using the {\sc Maple} software package) shows
\[
\begin{array}{rcl}
A_1&=&3+2b+2c+d+a\\
H_2&=&{a}^{2}b+{a}^{2}c+{a}^{2}d+3\,a{b}^{2}+7\,abc+4\,abd+3\,a{c}^{2}+4\,acd\\
                 &+&a{d}^{2}+2\,{b}^{3}+7\,{b}^{2}c
        +3\,{b}^{2}d+7\,b{c}^{2}+6\,bcd\\
&+&b{d}^{2}+2\,{c}^{3}+3\,{c}^{2}d
+c{d}^{2}+2\,{a}^{2}+10\,ba+9\,ca+6\,da+10\,{b}^{2}\\
&+&21\,cb+10\,db+9\,{c}^{2}+10\,cd+2\,{d}^{2}+8\,a+16\,b+14\,c+8\,d+8\\
 A_3&=&{b}^{2}c+b{c}^{2}+bcd+ba+ca+da+{b}^{2}+2\,cb+db+a+2\,b+c+d+1
\end{array}
\]
and the nonnegativity of $a,\ldots, d$ implies positivity of $A_1$, $H_2$ and $A_3$. Thus condition (TFII) holds.\\

Now we address global convergence. No explicit parameterization of $Y$ seems to be known. Nonetheless in \cite{gw2} the existence of exactly one stationary point in $K$ was shown, and also its linear asymptotic stability on $Y$ and global asymptotic stability on $K$. By Theorem \ref{1dprop} we obtain the desired convergence result.

\subsection{A two-dimensional slow manifold}
Some variants of Michaelis-Menten may lead to two-dimensional slow manifolds, depending on the parameters.  We consider again competitive inhibition with irreversible product formation (see subsection \ref{compinone}), but now with small parameters $k_1=\varepsilon k_1^*$, $k_{-1}=\varepsilon k_{-1}^*$ and $k_2=\varepsilon k_2^*$. The differential equation is
  \begin{align*}
   &\dot s=\varepsilon\left[k_{-1}^*c_1-k_1^*s(e_0-c_1-c_2)\right]\\
   &\dot c_1=\varepsilon\left[k_1^*s(e_0-c_1-c_2)-(k_{-1}^*+k_2^*)c_1\right]\\
   &\dot c_2=k_3(e_0-c_1-c_2)(i_0-c_2)-k_{-3}c_2
  \end{align*}
 on the (chemically relevant) positively invariant compact set \[L:=\{(s,c_1,c_2)\in {\mathbb R}^3_+,\ c_1+c_2\leq e_0,\ s+c_1\leq s_0,\ c_2\leq i_0\}.\] A short computation (see \cite{gwz}, p.~1175 f.) shows that (TFII) is satisfied, and that there exists a convergent Tikhonov-Fenichel reduction to
  \begin{align}
   &s'=k_{-1}^*c_1-k^*_1s(e_0-c_1-c_2)\label{inhib1}\\
   &c_1'=k_1^*s(e_0-c_1-c_2)-(k_{-1}^*+k_2^*)c_1\label{inhib2}\\
   &c_2'=\frac{-(i_0-c_2)[k_1^*s(e_0-c_1-c_2)-(k_{-1}^*+k_2^*)c_2]}{\kappa+e_0+i_0-c_1-2c_2}
  \end{align}
 in $L\cap Y$ with the two-dimensional asymptotic slow manifold \[Y:=\{(s,c_1,c_2)\in  {\mathbb R}^3_+,\ (e_0-c_1-c_2)(i_0-c_2)-\kappa c_2=0\}\]
 and $\kappa:=\frac{k_{-3}}{k_3}$. (Note that $\kappa+e_0+i_0-c_1-2c_2>0$ on $L\cap Y$.)\\ Thus, with
  \begin{equation}\label{inhibman}
   c_2=\vartheta(c_1):=\frac{\kappa+e_0+i_0-c_1-\sqrt{(\kappa+e_0+i_0-c_1)^2-4i_0(e_0-c_1)}}{2},
  \end{equation}
 $L\cap Y$ is contained in the graph of a function of $s$ and $c_1$, and it suffices to analyze \eqref{inhib1}--\eqref{inhib2} in the positively invariant compact set
  \[
   \widetilde L:=\{(s,c_1)\in  {\mathbb R}^2_+,\ c_1+\vartheta(c_1)\leq e_0,\ s+c_1\leq s_0,\, \vartheta(c_1)\leq i_0\}.
  \]
We will abbreviate this system as
\[
\begin{pmatrix} s\\ c_1\end{pmatrix}^\prime =q\left(\begin{pmatrix} s\\ c_1\end{pmatrix}\right).
\]
 It is easy to see that $(0,0)$ is the only stationary point in $\widetilde L$. We construct a suitable Lyapunov function. Let $\alpha>0$ and note that
\[
\begin{array}{rcccl}
\varphi_1&:=&s+c_1&\mbox{  satisfies  }& L_q(\varphi_1)= -k_2^*c_1,\\
\varphi_2&:=&\alpha s&\mbox{  satisfies  }& L_q(\varphi_2)= \alpha(k_{-1}^*+k_1^*s)c_1 -\alpha k_1^*s(e_0-\vartheta(c_1)).\\
\end{array}
\]
On $L\cap Y$ we have
$0\leq k_{-1}^*+k_1^*s\leq k_{-1}^*+k_1^*s_0$,
and furthermore $e_0-c_2>0$. The first of these assertions is obvious. To verify the second, note that $e_0-c_1-c_2\geq 0$ by stoichiometry, whence $e_0-c_2=0$ implies $c_1=0$ and, by the defining equation for $Y$,
\[
0=(e_0-c_1-c_2)(i_0-c_2)=\kappa c_2=\kappa e_0>0,
\]
a contradiction.\\
By compactness, there exists $\beta >0$ such that $e_0-c_2\geq \beta$ for all points in $L$. Now choose $\alpha>0$ such that
\[
\alpha\cdot(k_{-1}^*+k_1^*s_0)<k_2^*,
\]
and define $\varphi:=\varphi_1+\varphi_2$. Then, by the previous estimates we have on $\widetilde L$:
\[
\begin{array}{rcl}
   L_q(\varphi)(s,c_1)&\leq& -(k_2^*-\alpha\cdot(k_{-1}^*+k_1^*s_0))c_1 -\alpha k_1^*\beta s\\
                                &\leq & -\nu\cdot ((1+\alpha)s+c_1) =-\nu\cdot \varphi
\end{array}
  \]
for some $\nu>0$. Thus (LC) holds, and by Theorem \ref{hshsatz} convergence to the reduced system holds for all positive times.
\section{Concluding remarks}
We finish with a few observations, and some comments on possible extensions and generalizations.
\begin{itemize}
\item The results of the present paper will be hardly surprising to application-oriented readers (e.g.  with a background in biochemistry) although the underlying mathematical argument (essentially due to Hoppensteadt \cite{Hoppensteadt}) is far from trivial, and not easy to transfer to applications. This gap between intuition (where matters may seem obvious) and rigorous proofs (which may require an extensive technical build-up) can be observed quite frequently. The authors' main goal was to facilitate the applicability of Hoppensteadt's theorem to relevant settings.
\item By its nature Theorem \ref{hshsatz} is local, but the domain of attraction of $Y$ may properly contain $\widetilde K$. There remains, however, the question how fast a solution approaches the slow manifold. \\ In the setting of reaction equations it is known from the work of Horn and Jackson \cite{horja} and Feinberg \cite{feinberg} that global Lyapunov functions often exist; for instance this is the case for deficiency zero and complex balanced systems. Given a system with slow and fast reactions, a Lyapunov function for the fast subsystem could imply a condition akin to Hoppensteadt's condition (VII), with reasoning similar to (and extending) Lemma \ref{lyapunov}. But the Lyapunov functions from the cited papers do not generally satisfy the hypotheses of the Lemma; thus a case-by-case analysis would be in order. 
\item In Lemma \ref{1dlyapunov}, if $K$ is a neighborhood of $0$ then it suffices that $p$ changes sign at $0$, with the lowest (necessarily odd) order nonzero derivative at $0$ being negative. By analogous estimates one obtains conditions (ii) and (iii) of Lemma \ref{lyapunov}, with exponents $a>2$ and $k>1$. In order to transfer this to systems with one-dimensional slow manifold one would have to generalize Proposition \ref{1dprop} by requiring suitable properties of the Poincar\'e-Dulac normal form at $z$. 

\end{itemize}

\section{Appendix: Hoppensteadt's conditions}\label{appsec}
For the reader's convenience we recall here the conditions and the main result from Hoppensteadt's original paper \cite{Hoppensteadt}.
Recall the notation $S_R$ from \eqref{tubeone}.
Hoppensteadt considers a non-autonomous system that is given in Tikhonov standard form
  \begin{align}
   &y_1'= f(\tau,y_1,y_2,\varepsilon)\label{tnfhs1na}\\
   &y_2'=\varepsilon^{-1}g(\tau,y_1,y_2,\varepsilon)\label{tnfhs2na}
  \end{align}
with $f$ and $g$ defined on an open set 
\[[0,\infty)\times D \times [0,\varepsilon_0)\subseteq [0,\infty)\times\mathbb R^s\times \mathbb R^r\times [0,\varepsilon_0)\to \mathbb R^r\] 
which satisfies $S_R\subseteq D$ for some $R>0$, and $f,\,g$ having values in $\mathbb R^s$ and $\mathbb R^r$, respectively.
Assume that the following conditions hold:
  \begin{itemize}
   \item[(I)] The system 
    \begin{align}
     y_1'&=f(\tau,y_1,y_2,0) \label{redSystem1na}\\
     0&=g(\tau,y_1,y_2,0)\label{redSystem2na}
    \end{align}
 admits a solution $Y:\,[0,\infty)\to \mathbb R^{s+r}$,  $\tau\mapsto Y(\tau)$. With a suitable transformation of \eqref{redSystem1na}--\eqref{redSystem2na} one may assume that $\widetilde Y\equiv0$ is a solution of the transformed system. From here on, it will be assumed that \eqref{redSystem1na}--\eqref{redSystem2na} admits the solution $\widetilde Y\equiv 0$.
   \item[(II)] The functions $f$, $g$ and their partial derivatives with respect to $\tau$, $y_1$, $y_2$ respectively satisfy \[f,g,D_1f,D_2f,\partial_{\tau}g,D_1g,D_2g \in C([0,\infty)\times S_R\times [0,\varepsilon_0]).\] 
   \item[(III)] There exists an  isolated and  bounded  $C^2$-solution $Y_2=Y_2(\tau,y_1)$ of the implicit equation
    \[
     g(\tau,y_1,Y_2(\tau,y_1),0)=0
    \]
   for all $\tau\in [0,\infty)$ and $y_1\in S_1$. By a transformation $\widetilde y_1=y_1$ and $\widetilde y_2+Y(\tau,y_1)=y_2$,  one may obtain that $Y_2(\tau,\widetilde y_1)=0$ for all $(\tau,\widetilde y_1)\in [0,\infty)\times S_1$. This will be assumed for the original system \eqref{tnfhs1na}--\eqref{tnfhs2na} in the following.
   \item[(IV)] $f(\cdot,\cdot,0,0)$ is uniformly continuous in $[0,\infty)\times S_{1,R}$, and moreover $f(\cdot,\cdot,0,0)$ and $D_1f(\cdot,\cdot,0,0)$ are bounded in $[0,\infty)\times S_{1,R}$. 
   \item[(V)] $g(\cdot,\cdot,\cdot,0)$ is uniformly continuous in $[0,\infty)\times S_R$, and moreover $g(\cdot,\cdot,\cdot,0)$, $\partial_{\tau}g(\cdot,\cdot,\cdot,0)$, $D_1g(\cdot,\cdot,\cdot,0)$ $D_2g(\cdot,\cdot,\cdot,0)$ are bounded in $[0,\infty)\times S_R$.
   \item[(VI)] The solution $\widetilde Y\equiv0$ of
    \begin{equation}\label{(D)}
     y'=f(\tau,y,0,0)
    \end{equation}
  is uniformly asymptotically stable in the following sense: There exist a continuous, strictly increasing function \[d\colon [0,\infty)\to[0,\infty)\quad \text{with }d(0)=0\] and a continuous, strictly decreasing function $\sigma\colon [0,\infty)\to[0,\infty)$ with $\lim_{s\to\infty}\sigma(s)=0$, such that for every solution $\Phi(\tau,z)$ of \eqref{(D)} with initial value $y_1(0)=z\in S_{1,R}$  and all $\tau\geq0$ one has
    \[
     \abs{\Phi(\tau,z)}_1\leq d(\abs{z}_1)\cdot\sigma(\tau).
    \]
   
   \item[(VII)] For all $(\alpha,\beta)\in[0,\infty)\times S_{1,R}$  the solution $\widetilde Y\equiv0$ of
  \[
   \dot x=g(\alpha,\beta,x,0)
  \]
is uniformly asymptotically stable in the following sense:  There exist a continuous, strictly increasing function \[e\colon [0,\infty)\to[0,\infty)\quad \text{with }e(0)=0\] and a continuous, strictly decreasing function $\rho\colon [0,\infty)\to[0,\infty)$ with $\lim_{s\to\infty}\rho(s)=0$, such that for every solution $\Psi(t,x_{0};\alpha,\beta)$ of the equation with initial value $x(0)=x_{0}\in S_{2,R}$ and parameters $(\alpha,\beta)\in[0,\infty)\times S_{1,R}$ and all $t\geq0$ one has
    \[
     \abs{\Psi(t,x_0;\alpha,\beta)}_1\leq e(\abs{x_0}_1)\cdot\rho(t).
    \]
  \end{itemize}

Given these assumptions, Hoppensteadt's main result \cite{Hoppensteadt} can be stated as follows:
\begin{theorem}\label{HoppensteadtSatz}
There exists a compact neighborhood  $K\subset S_R$ of $0$ and $\varepsilon_0^*\in (0,\varepsilon_0)$ such that the solution $\Phi(t,y_0,\varepsilon)$ of \eqref{tnfhs1na}--\eqref{tnfhs2na} with initial value $y(0)=y_0:=(y_{1,0},y_{2,0})\in K$ at $\tau=0$ exists for all positive times provided that $0<\varepsilon<\varepsilon_0^*$. Moreover $\Phi(t,y_0,\varepsilon)$ converges uniformly on all closed subsets of $(0,\infty)$ towards the solution of \eqref{redSystem1na}--\eqref{redSystem2na} with initial value $y_1(0)=y_{1,0}$, as $\varepsilon\to0$.
\end{theorem}


\begin{thebibliography}{99}
\bibitem{Amann} H.~Amann: {\it Ordinary Differential Equations. An Introduction to Nonlinear Analysis}. Walter de Gruyter, Berlin - New York (1990).

\bibitem{bashi} J.~Back, H.~Shim: {\it Adding robustness to nominal output-feedback controllers for uncertain nonlinear systems: A nonlinear version of disturbance observer.} Automatica {\bf 44}, 2528--2537 (2008).

\bibitem{briggshaldane}
G.~E. Briggs and J.~B.~S. Haldane.
\newblock A note on the kinetics of enzyme action.
\newblock {\em Biochem. J.}, 19:338--339, 1925.

\bibitem{cavnat} A.~Cavallo, C.~Natale: {\it Output feedback control based on a high-order sliding manifold approach.} IEEE Transactions on Automatic Control {\bf 48}, 469--472 (2003).

\bibitem{feinberg} M.~Feinberg: {\it The existence and uniqueness of steady states for a class of chemical reaction networks.}
Arch. Ration. Mech. Anal. {\bf 132}, 311--370 (1995).

\bibitem{fenichel} N.~Fenichel: {\it Geometric singular perturbation theory for ordinary
              differential equations}. {J. Differential Equations} {\bf31}, {53--98} (1979).

\bibitem{gantmacher} F.R.~Gantmacher: {\it Applications of the theory of matrices.} Dover Publ., Mineola, NY (2005).
   
\bibitem{gswz} A.~Goeke, C.~ Schilli, S.~Walcher, E.~ Zerz: {\it Computing quasi-steady state reductions}. {J. Math. Chem.} {\bf 50}, {1495--1513} (2012).


\bibitem {gw2} A.~Goeke, S.~Walcher: {\it A constructive approach to quasi-steady state reductions}. {J. Math. Chem.} {\bf 52}, {2596--2626} (2014).

\bibitem {gwz} A.~Goeke, S.~Walcher, E.~ Zerz: {\it Determining ``small parameters'' for quasi-steady state}. {J. Differential Equations} {\bf 259}, {1149--1180} (2015).
 


\bibitem{hirsch} M.W.~Hirsch: {\it Differential Topology.} Springer, New York  (1976).

\bibitem{horja} F.~Horn, R.~Jackson: {\it General mass action kinetics.} Arch. Ration. Mech. Anal. {\bf 47}, 8--116 (1972).

\bibitem{Hoppensteadt} F.C.~ Hoppensteadt: {\it Singular Perturbations on the Infinite Interval}.  {Trans. Amer. Math. Soc.} {\bf 123},  521--535 (1966).

\bibitem{KeenerSneyd} J.~ Keener, J.~ Sneyd: {\it Mathematical physiology I: Cellular physiology}, Second Ed. Springer-Verlag, New York (2009).

\bibitem{laxdiss} C.~Lax: {\it Analyse und asymptotische Analyse von Kompartimentsystemen.} Doctoral dissertation, RWTH Aachen (2016). 


\bibitem{michaelismenten}
L.~Michaelis and M.~L. Menten.
\newblock {Die Kinetik der Invertinwirkung}.
\newblock {\em Biochem. Z.}, 49:333--369, 1913.

\bibitem{milnor} J.W.~Milnor: {\it Topology from the Differentiable Viewpoint.} Princeton University Press, Princeton (1997).
 

\bibitem{nestruev2003smooth} J.~Nestruev: {\it Smooth manifolds and observables}. Springer, New York (2003).

\bibitem{nw11} L.~Noethen, S.~ Walcher: {\it Tikhonov's theorem and quasi-steady state}.
 {Discrete Contin. Dyn. Syst. Ser. B} {\bf 16}, {945--961} (2011).

\bibitem{SSl} L.A.~Segel, M.~Slemrod: {\it The quasi-steady-state assumption: A case study in perturbation.} SIAM Review {\bf 31}, 446 - 477 (1989).

\bibitem{selma}K.~Seliger: {\it Singul\"are St\"orungen auf unbeschr\"ankten Intervallen.} Master's thesis, RWTH Aachen (2015).
 
\bibitem{sti} M.~Stiefenhofer: {\it Quasi-steady-state approximation for chemical reaction networks}. {J. Math. Biol.} {\bf 36}, {593--609} (1998).

\bibitem{teeletal} A.R.~Teel, L.~Moreau, D.~Nesic: {\it A unified framework for input-to-state-stability in systems with two time scales.} IEEE Transactions on Automatic Control {\bf 48}, 1526--1544 (2003).

\bibitem{tikh} A.N.~Tikhonov:  {\it Systems of differential equations containing a small parameter multiplying the derivative (in Russian)}. {Math. Sb.} {\bf 31}, {575--586} (1952).


\bibitem{Walter} W.~Walter: {\it Ordinary differential equations.} Springer-Verlag, New York (1998).


\end{thebibliography}

\end{document}